\documentclass[letterpaper, 10pt]{article}
\usepackage[psamsfonts]{amssymb}
\usepackage{amsmath, amsthm, anysize, enumerate, color, lineno}
\usepackage[pdftex]{graphicx}

\newtheorem{theorem}{Theorem}[section]

\newtheorem{problem}[theorem]{Problem}
\newtheorem{corollary}[theorem]{Corollary}
\newtheorem{conjecture}[theorem]{Conjecture}
\newtheorem{lemma}[theorem]{Lemma}
\newtheorem{proposition}[theorem]{Proposition}
\newenvironment{proof_of}[1]{\noindent {\bf Proof of #1:}
	\hspace*{1mm}}{\hspace*{\fill} $\qedsymbol$ }

\definecolor{darkgreen}{RGB}{0,180,0}

\theoremstyle{definition}
\newtheorem{remark}[theorem]{Remark}
\newtheorem{example}[theorem]{Example}

\newcommand{\E}{\mathbb{E}}
\newcommand{\R}{\mathbb{R}}
\newcommand{\N}{\mathbb{N}}

\renewcommand{\P}{\mathbb{P}}

\newcommand{\T}{\mathcal{T}}
\newcommand{\D}{\mathcal{D}}
\newcommand{\sign}{\text{sign}}

\newcommand{\ones}{\mathbf{1}}
\newcommand{\eye}{I}
\newcommand{\basis}{\mbox{\boldmath$e$}}
\newcommand{\Basis}{\mbox{\boldmath$e$}}
\newcommand{\rank}{{\rm rank}\,}
\def\noqed{\renewcommand{\qedsymbol}{}}

\begin{document}

\title{Inverses of symmetric, diagonally dominant positive matrices}

\author{Christopher~J.~Hillar\thanks{Redwood Center for Theoretical
    Neuroscience, \texttt{chillar@msri.org}; partially supported by NSF grant IIS-0917342 and 
    an NSF All-Institutes Postdoctoral Fellowship administered by the
    Mathematical Sciences Research Institute through its core grant
    DMS-0441170.}, Shaowei Lin\thanks{Department of Mathematics, \texttt{shaowei@math.berkeley.edu}; supported by the DARPA
    Deep Learning program (FA8650-10-C-7020).}, Andre
  Wibisono\thanks{Department of Electrical Engineering and Computer
    Science, \texttt{wibisono@eecs.berkeley.edu}.} \\ $ $\\University of California, Berkeley}

\date \today

\maketitle

\begin{abstract}
We prove tight bounds for the $\infty$-norm of the inverse of symmetric, diagonally dominant positive matrices. We also prove a new lower-bound form of Hadamard's inequality for the determinant of diagonally dominant positive matrices and an improved upper bound for diagonally balanced positive matrices.  Applications of our results include numerical stability for linear systems, bounds on inverses of differentiable functions, and consistency of the maximum likelihood equations for maximum entropy graph distributions.
\end{abstract}

\section{Introduction}

An $n\times n$ real matrix $J$ is {\em diagonally dominant} if
\begin{equation*}
\Delta_i(J) := |J_{ii}| - \sum_{j \neq i} |J_{ij}| \geq 0, \quad \text{for } i = 1,\dots,n.
\end{equation*}
Irreducible, diagonally dominant matrices are always invertible, and such matrices arise often in theory and applications.  As a recent example, the work of Spielman and Teng \cite{spielman-teng04, spielman-teng06} gives algorithms to solve symmetric, diagonally dominant linear systems in nearly-linear time in the input size, a fundamental advance in algorithmic complexity theory and numerical computation.  These systems are important since they arise naturally in many practical applications of linear algebra to graph theory \cite{spielman10}.  In this paper, we study mathematical properties of the inverse and determinant of symmetric diagonally dominant matrices that have only positive entries.

By far, the most useful information about the inverses of such matrices in applications are bounds on their norms, so we discuss these properties first. A classical result of Varah~\cite{varah1975lower} states that if $J$ is \textit{strictly diagonally dominant}, i.e. if $\Delta_i(J) > 0$ for $1 \leq i \leq n$, then the inverse of $J$ satisfies the bound:
\begin{equation*}
\|J^{-1}\|_{\infty} \leq \max_{1 \leq i \leq n} \: \frac{1}{\Delta_i(J)}.
\end{equation*}
Here, $\| \cdot \|_{\infty}$ is the maximum absolute row sum of a matrix, which is the matrix norm induced by the infinity norm $|\cdot|_{\infty}$ on vectors in $\mathbb R^n$.
Generalizations of this basic estimate can be found in~\cite{varga1976diagonal},~\cite{shivakumar1996two}, and~\cite{li2008}, but all involve the quantity $\max_{1 \leq i \leq n} 1/\Delta_i(J)$. In practice, however, one sometimes requires bounds when $\Delta_i(J) = 0$ for some $i$, in which case the estimates appearing in~\cite{varah1975lower,varga1976diagonal,shivakumar1996two, li2008} do not apply. A particularly interesting case is when $\Delta_i(J) = 0$ for all $i$; we call such matrices {\em diagonally balanced}.

We prove a tight bound on $\|J^{-1}\|_\infty$ for symmetric diagonally
dominant $J$ with positive entries that is independent of the
quantities $\Delta_i(J)$, and thus also of the maximum entry of $J$.
Let $S = (n-2)\eye_n + \ones_n \ones_n^\top$ be the diagonally
balanced matrix whose off-diagonal entries are all equal to $1$, and
recall the \textit{Loewner partial ordering} on symmetric matrices: $A
\succeq B$ means that $A-B$ is positive semidefinite. We shall also write $A
\geq B$ if $A-B$ is a nonnegative matrix. In
Lemma~\ref{Lem:EigBalanced} below, we show that if $\ell >0$ and $J$ is a symmetric
diagonally dominant matrix satisfying $J \geq \ell S$, then $J \succeq
\ell S \succ 0$; in particular, $J$ is invertible. Throughout this
paper, $\eye_n$ and $\ones_n$ denote the $n \times n$ identity matrix and the $n$-dimensional column vector consisting of all ones, respectively. We also write $\eye$ and $\ones$ if the dimension $n$ is understood.

The following is our first main result. 

\begin{theorem}\label{Thm:Main}
Let $n \geq 3$. For any symmetric diagonally dominant matrix $J$ with $J_{ij} \geq \ell > 0$, we have
\begin{equation*}
\|J^{-1}\|_\infty \leq \frac{1}{\ell} \|S^{-1}\|_\infty = \frac{3n-4}{2\ell(n-2)(n-1)}.
\end{equation*}
Moreover, equality is achieved if and only if $J = \ell S$.
\end{theorem}

\begin{remark}\label{Ex:NonSym}
Theorem~\ref{Thm:Main} fails to hold if we relax the assumption that $J$ be symmetric.  For $t \geq 0$, consider the following diagonally balanced matrices and their inverses:
\begin{equation*}
J_t =
\begin{pmatrix} 2+t & 1 & 1+t \\ 1 & 2+t & 1+t \\ 1 & 1 & 2 \end{pmatrix};
\quad  J_t^{-1} = \frac{1}{4}
\begin{pmatrix}
\frac{t+3}{t+1} & \frac{t-1}{t+1} & -t-1 \vspace{1mm}\\
\frac{t-1}{t+1} & \frac{t+3}{t+1} & -t-1 \vspace{1mm}\\
-1 & -1 & t+3
\end{pmatrix}.
\end{equation*}
We have $\|J_t^{-1}\|_\infty \to \infty$ as $t \to \infty$. \qed
\end{remark}

In other words, the map $J \mapsto \|J^{-1}\|_\infty$ over the (translated) cone of symmetric diagonally dominant  matrices $J \geq \ell S$ is maximized uniquely at the point of the cone; i.e., when $J = \ell S$.  If the off-diagonal entries of $J$ are
bounded above by $m$ and the largest of the diagonal dominances $\Delta_i(J)$
is $\delta$, we also have the trivial lower bound:
\[\frac{1}{2m(n-1)+\delta} \leq \|J^{-1}\|_\infty,\]
which follows from  submultiplicativity of the matrix norm $\| \cdot \|_\infty$.
Therefore, the value of $\|J^{-1}\|_\infty = \Theta(\frac{1}{n})$ is tightly
constrained for bounded $\ell, m$ and $\delta$.

We now probe some of the difficulty of Theorem \ref{Thm:Main} by first deriving an estimate using standard bounds in matrix analysis.
The relation $J \succeq \ell S$ is equivalent to $J^{-1} \preceq (\ell S)^{-1} = \frac{1}{\ell} S^{-1}$~\cite[Corollary~7.7.4]{HornJohnson}, and therefore by a basic inequality~\cite[p.\ 214, Ex.\ 14]{HornJohnsonTopic}, we have $\|J^{-1}\| \leq \frac{1}{\ell} \|S^{-1}\|$ for any unitarily invariant matrix norm $\|\cdot\|$, such as the spectral $\| \cdot \|_{2}$, Frobenius, or Ky-Fan norms. It follows, for instance, that
\begin{equation}\label{Eq:SpectralInftyBound}
\|J^{-1}\|_\infty \leq \sqrt{n} \: \|J^{-1}\|_2 \leq \frac{\sqrt{n}}{\ell} \|S^{-1}\|_2 = \frac{\sqrt{n}}{(n-2)\ell}.
\end{equation}
However, this bound is $O(\frac{1}{\sqrt{n}})$, whereas the bound given in
Theorem~\ref{Thm:Main} is $O(\frac{1}{n})$. In some applications this
difference can be crucial. For instance,
we explain in Section \ref{Sec:App} how Theorem~\ref{Thm:Main} proves the consistency of the maximum likelihood estimator for some random graph distributions inspired by neuroscience \cite{HilWib}.

Another standard approach to proving norm estimates such as the one in Theorem~\ref{Thm:Main} is a perturbation analysis. More specifically, given a symmetric diagonally dominant $J$ with entries bounded below by $\ell$, one tries to replace each entry $J_{ij}$ by $\ell$ and prove that the norm of the inverse of the resulting matrix is larger. However, such a method will not succeed, even in the balanced case, as the following examples demonstrate.

\begin{example}
Consider the  following two balanced matrices:
\begin{equation*}
J = \begin{pmatrix}
12 & 4 & 1 & 7 \\
4 & 9 & 3 & 2 \\
1 & 3 & 7 & 3 \\
7 & 2 & 3 & 12
\end{pmatrix}
\ \ \text{and} \  \ 
H = \begin{pmatrix}
9 & 1 & 1 & 7 \\
1 & 6 & 3 & 2 \\
1 & 3 & 7 & 3 \\
7 & 2 & 3 & 12
\end{pmatrix}.
\end{equation*}
Here, $H$ is obtained from $J$ by changing its $(1,2)$-entry to $1$ and then keeping the resulting matrix balanced. A direct computation shows that $\|H^{-1}\|_\infty < \|J^{-1}\|_\infty$. As another example, the following two matrices:
\begin{equation*}
J = \begin{pmatrix}
3 & 2 & 1 \\ 2 & 3 & 1 \\ 1 & 1 & 2
\end{pmatrix}
\ \ \text{and} \ \
H = \begin{pmatrix}
3 & 1 & 1 \\ 1 & 3 & 1 \\ 1 & 1 & 2
\end{pmatrix} \ \ 
\end{equation*}
also have $\|H^{-1}\|_\infty < \|J^{-1}\|_\infty$. Here, the
$(1,2)$-entry was changed \emph{without} keeping the matrix balanced.
\qed
\end{example}

We next describe an interesting special case revealing some surprising combinatorics underlying Theorem \ref{Thm:Main}.
Let $P$ be a symmetric diagonally dominant matrix with $P_{ij} \in \{0,1\}$ and $\Delta_i(P) \in \{0,2\}$. 
Each such matrix $P$ is a {\em signless Laplacian} of an undirected, unweighted graph $G$, possibly with self-loops. 
The limits
\begin{equation}\label{eqn:Nlim}
N = \lim_{t \to \infty} (S + tP)^{-1}
\end{equation}
form special cases of Theorem \ref{Thm:Main}, and we compute them explicitly in Section~\ref{Sec:Infinity}. As we shall see, they are an essential calculation for our proof.  The matrices $N$ are determined by the bipartition structure of the connected components of $G$.  For instance, if $G$ is connected and not bipartite, the limit (\ref{eqn:Nlim}) is the zero matrix (see Corollary~\ref{Thm:blockconstants}).  Example \ref{Ex:Graphs} and Figure \ref{fig_graph_limits} below contain more interesting cases.  For some recent work on the general eigenstructure of signless Laplacians, we refer the reader to \cite{cvetkov2010towards3} and the references therein.

\begin{example}\label{Ex:Graphs}
Consider the chain graph $G$ with edges $\{1,n\}$ and $\{i,i+1\}$ for $i = 1,\dots,n-1$. If $n$ is odd then $N = 0$ since $G$ is not bipartite, while if $n$ is even, the limit $N$ has alternating entries:
\begin{equation*}
N_{ij} = \frac{(-1)^{i+j}}{n(n-2)}. 
\end{equation*}
As another example, consider the star graph $G$, which has edges $\{1,i\}$ for $i = 2,\dots,n$.  In this case,
\begin{flalign*}
&&N_{1i} = N_{i1} = -\frac{1}{2(n-1)(n-2)}, \; \text{for } i = 2,\dots,n;
\quad\text{ and }\quad
N_{ij} = \frac{1}{2(n-1)(n-2)}, \; \text{otherwise}. &&\qed
\end{flalign*}
\end{example}

We next describe our second collection of main results concerning the determinant and adjugate of positive, diagonally dominant symmetric matrices $J$.  Recall the formula:
\[J^{-1}  = \frac{J^{\star}}{\det(J)},\]
in which $J^{\star}$ is the adjugate (or classical adjoint) of $J$.  To prove Theorem \ref{Thm:Main}, one might first try to bound $\|J^{\star}\|_{\infty}$ from above and $\det(J)$ from below, preferably separately.  We first focus on estimating the latter, which is already nontrivial.

It is classical that the determinant of a positive semidefinite matrix $A$ is bounded above by the product of its diagonal entries:
\[ 0 \leq  \det(A) \leq \prod_{i=1}^n A_{ii}.\]
This well-known result is sometimes called Hadamard's inequality~\cite[Theorem~7.8.1]{HornJohnson}.  A lower bound of this form, however, is not possible without additional assumptions.  Surprisingly, there is such an inequality when $J$ is diagonally dominant with positive entries; and, when $J$ is also balanced, an improved upper bound.

\begin{theorem}\label{Thm:DetBound}
Let $n \geq 3$, and let $J$ be an $n \times n$ symmetric matrix with off-diagonal entries $m \geq J_{ij} \geq \ell > 0$.
\begin{enumerate}[{\bf (a)}]
  \item If $J$ is diagonally dominant, then
\begin{equation*}
\frac{\det(J)}{\prod_{i=1}^n J_{ii}} \geq \left(1-\frac{1}{2(n-2)} \sqrt{\frac{m}{\ell}} \left(1 + \frac{m}{\ell} \right) \right)^{n-1} \to \ \exp \left(-\frac{1}{2} \sqrt{\frac{m}{\ell}} \left( 1 + \frac{m}{\ell} \right)\right) \  \text{ as $n \to \infty$}.
\end{equation*}
\item If $J$ is diagonally balanced, then  
\begin{equation*}
\frac{\det(J)}{\prod_{i=1}^n J_{ii}} \leq \exp\left(-\frac{\ell^2}{4m^2}\right).
\end{equation*}
\end{enumerate}
\end{theorem}

\begin{figure}
	\begin{center}
\includegraphics[width=6.2in]{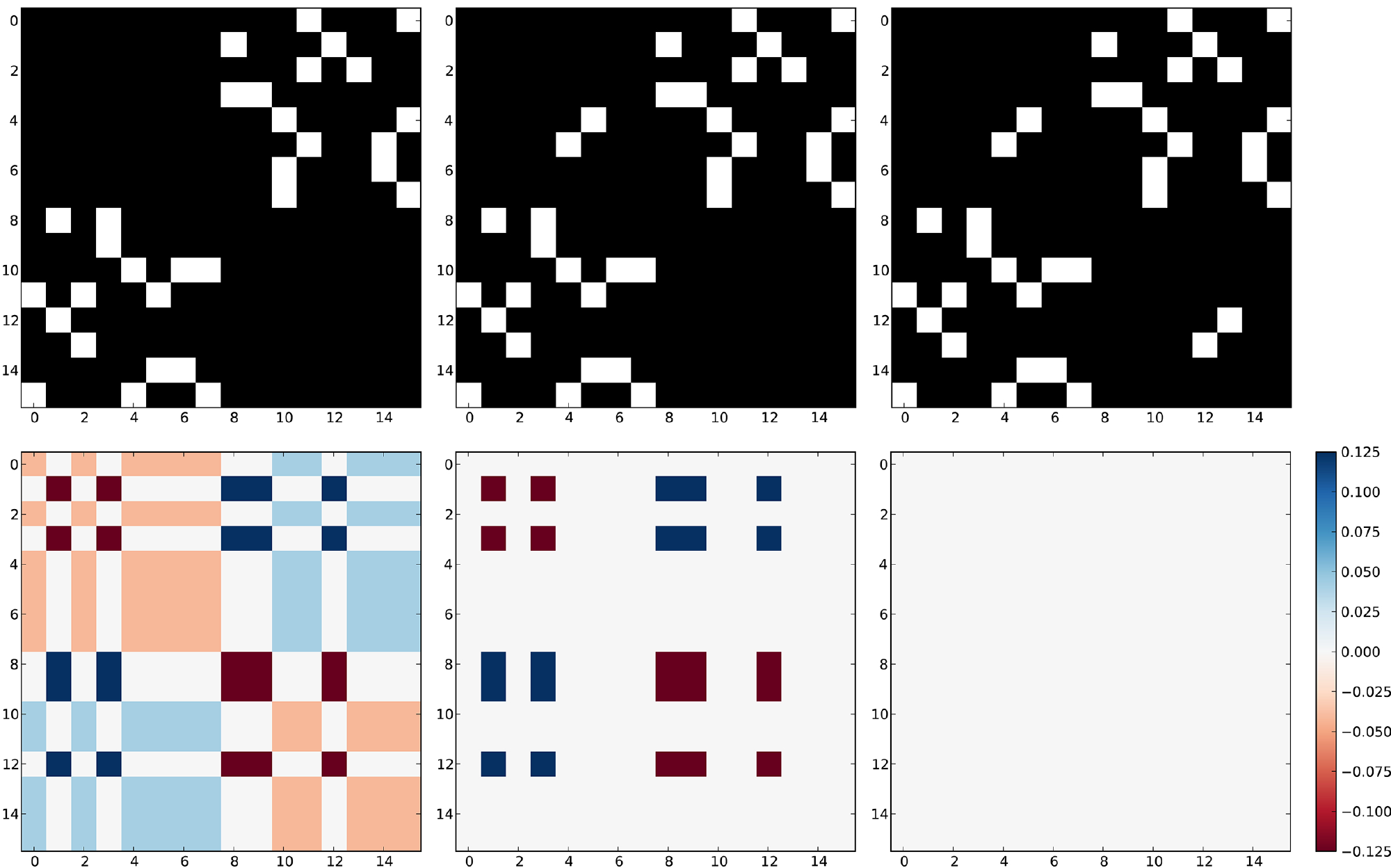}
\caption{\small{\textbf{Limits (\ref{eqn:Nlim}) for different signless Laplacian matrices $P$}.  The adjacency matrices of three graphs on $16$ vertices are represented in the top panes of the figure above.  The two bipartite graphs on the top left have structured $(S+tP)^{-1}$ for large $t$, whereas for the non-bipartite graph on the top right, this limit (bottom right) is zero.}}
\label{fig_graph_limits}
\end{center}
\end{figure}

The bounds above depend on the largest off-diagonal entry of $J$ (in an essential way for ({\bf a}); see Example \ref{NonupperBndEx}), and thus are ill-adapted to prove Theorem \ref{Thm:Main}. For instance, combining Theorem \ref{Thm:DetBound} (\textbf{a}) with Hadamard's inequality applied to the positive definite $J^{\star}$ in the obvious way gives estimates which are worse than (\ref{Eq:SpectralInftyBound}).  Nevertheless, Theorem~\ref{Thm:DetBound} should be of independent interest, and we prove it in Section~\ref{Sec:Hadamard}.

As an immediate application of Theorems \ref{Thm:Main} and \ref{Thm:DetBound}, we obtain a bound on $J^{\star}$ for balanced, positive $J$.

\begin{corollary}\label{adjugateCor}
Let $n \geq 3$, and let $J$ be a symmetric
$n \times n$ matrix $J$ with off-diagonal entries $m \geq  J_{ij} \geq
\ell > 0$.  If $J$ is diagonally balanced, then the adjugate $J^{\star}$ satisfies the bound:
\[   \frac{\|J^{\star}\|_{\infty}}{\prod_{i=1}^n J_{ii}} \leq  \frac{3n-4}{2\ell(n-2)(n-1)} e^{-\ell^2/(4m^2)}.\]
\end{corollary}

We finish this introduction with a brief overview of how our main
results are proved.  Theorem~\ref{Thm:Main} will be generalized in Theorem~\ref{Thm:GeneralMain} where we consider diagonally dominant matrices $J \geq
S(\alpha,\ell) := \alpha \eye_n + \ell \ones_n \ones_n^\top$ with $\alpha \geq (n-2)\ell > 0$. We break up the proof of this general theorem into three main steps in Sections~\ref{Sec:CornerCase} to \ref{Sec:MainProof}, where we write $S(\alpha,\ell)$ as $S$ for simplicity. The first step considers the problem of maximizing $\|J^{-1}\|_\infty$ over symmetric diagonally dominant $J$
with $J_{ij}$ and $\Delta_i(J)$ in some finite intervals
(Section~\ref{Sec:CornerCase}). In this case, the maximum is achieved
when $J$ is on the \emph{corners} of the space; namely, when $J_{ij}$ and
$\Delta_i(J)$ are one of the endpoints of the finite intervals. In the
second step (Section~\ref{Sec:Infinity}), we analyze the behavior of
corner matrices at infinity and show that the limit $\|(S +
tP)^{-1}\|_\infty$ as $t \to \infty$ when $P$ is a signless Laplacian
($P_{ij} \in \{0,1\}$ and $\Delta_i(P) \in \{0,2\}$) is at most
$\|S^{-1}\|_\infty$. Combined with the first step, this verifies that the
matrix $S$ maximizes $\|J^{-1}\|_\infty$ over the space of symmetric
diagonally dominant matrices $J \geq S$. The remainder of the argument deals with the behavior of $\|J^{-1}\|_\infty$ near $S$ to show that $S$ is indeed the unique maximizer (Section~\ref{Sec:SmallT}).  All three steps are combined in Section \ref{Sec:MainProof}.  
The inequalities of Theorem \ref{Thm:DetBound} are proved using a block matrix factorization in Section~\ref{Sec:Hadamard}. 

Finally, we conclude with a brief discussion of  open questions in Section~\ref{Sec:Problems}.

\section{Applications}\label{Sec:App}

The ($\infty$-norm) condition number $\kappa_{\infty}(A) = \|A\|_{\infty}\|A^{-1}\|_{\infty}$ of a matrix $A$ plays an important role in numerical linear algebra.  For instance, the relative error $|x-\hat{x}|_{\infty}/|x|_{\infty}$ of an approximate solution $\hat{x}$ to a set of  linear equations $Ax = b$ is bounded by the product of $\kappa_{\infty}(A)$ and the relative size of the residual, $|b - A\hat{x}|_\infty/|b|_{\infty}$ (e.g., \cite[p.\ 338]{HornJohnson}).  Directly from Theorem \ref{Thm:Main}, we may bound the condition number of a positive, diagonally dominant symmetric matrix.  Thus, numerical linear computation involving such matrices is well-behaved.

\begin{corollary}
The condition number $\kappa_{\infty}(A)$ of a positive, diagonally dominant symmetric $n \times n$ matrix $A$ with largest off-diagonal entry $m$, smallest entry $\ell$, and largest diagonal dominance $\delta$ satisfies: \[\kappa_{\infty}(A) \leq \frac{(2m(n-1)+\delta)(3n-4)}{2\ell(n-2)(n-1)}.\]
In particular, the condition number $\kappa_{\infty}(A)$ is always bounded above by $3m/\ell$ for large $n$.
\end{corollary}

We next discuss another application of Theorem \ref{Thm:Main} to the numerical stability of inverses of a large family of functions.
Let $U$ be a convex open subset of $\R^n$, $n \geq 3$, and consider the map $F \colon U \to \R^n$ given by
\begin{equation}\label{F_map}
F(x) = (F_1(x), \dots, F_n(x)), \ F_i(x) = \sum_{j=1}^n g_{ij}(x_i + x_j),
\end{equation}
in which $g_{ij} \colon \R \to \R$ are any continuously differentiable functions with $g_{ij} = g_{ji}$, $g'_{ij}(x_i + x_j) \geq \ell > 0$, and $g'_{ii}(2x_i) \geq 0$ for all $1 \leq i,j \leq n$ and $x \in U$. 

\begin{theorem}\label{inv_F_cor}
For any $d, \widehat{d} \in F(U)$, we have:
\begin{equation*}
|F^{-1}(d) - F^{-1}(\widehat d)|_\infty \leq \frac{3n-4}{2\ell(n-1)(n-2)} \: |d -\widehat{d}|_\infty.
\end{equation*}
\end{theorem}
\begin{proof}
The function $F$ is continuously differentiable on $U$, and its Jacobian $J =  J_F(x)$ satisfies:
\begin{equation*}
J_{ij}(x) = \frac{\partial F_i(x)}{\partial x_j} = g'_{ij}(x_i + x_j) \geq \ell, \quad
J_{ii}(x) = \frac{\partial F_i(x)}{\partial x_i} = 2g'_{ii}(2x_i) + \sum_{j \neq i} g'_{ij}(x_i + x_j) \geq \sum_{j \neq i} J_{ij}(x).
\end{equation*}
In particular, $J$ is symmetric and diagonally dominant with off-diagonal entries bounded below by $\ell > 0$.

By the mean-value theorem for vector-valued functions~\cite[p.~341]{Lang}, for each pair $x,y \in U$, we can write:
\begin{equation}\label{FJeqn}
F(x) - F(y) = \widetilde J (x-y),
\end{equation}
in which the matrix $\widetilde J = \int_0^1 J(tx + (1-t)y) \: dt$ is the element-wise average of the Jacobians on the line segment between $x$ and $y$. Thus, $\widetilde J$ is also symmetric and diagonally dominant with $\widetilde J_{ij} \geq \ell$. In particular, $\widetilde J$ is invertible, which shows that the map $F$ is invertible on $U$. By substituting $d = F(x)$ and $\widehat d = F(y)$ into (\ref{FJeqn}), inverting $\widetilde J$, and applying Theorem~\ref{Thm:Main} to $\widetilde J^{-1}$, we arrive at the inequality stated in the theorem.
\end{proof}

We close this section by explaining an application of Theorem \ref{inv_F_cor} to probability and statistics.  In~\cite{HilWib}, the maximum entropy distribution on (undirected) weighted graphs $A = (A_{ij})_{i,j=1}^n$ given an expected vertex degree sequence $d = (d_1,\ldots,d_n) \in \R^n$ is studied, extending the work of Chatterjee, Diaconis, and Sly~\cite{diaconis2011} in the case of unweighted graphs.  When the graphs have edges in $[0,\infty)$, this distribution $\P_\theta$ has random weights $A_{ij}$ that are independent exponential  variables with $\E[A_{ij}] = 1/(\theta_i + \theta_j) > 0$, where the parameters $\theta = (\theta_1,\ldots,\theta_n) \in \mathbb R^n$ are chosen so that the expected degree sequence of the distribution is equal to $d$.

Given $\theta$ with $1/m \leq (\theta_i + \theta_j)^2 \leq 1/\ell$ for $i \neq j$, suppose we draw a sample graph $\widehat G \sim \P_\theta$ with  weights $\widehat  A = (\widehat A_{ij})$, and let $\widehat d_i = \sum_{j \neq i} \widehat  A_{ij}$ be the degree sequence of $\widehat G$. The maximum likelihood estimator $\widehat \theta$ for $\theta$ solves the moment-matching equations (called the \textit{retina equations} in \cite{HilWib,SturmEntDisc}):
\begin{equation}\label{mom_match_eqn}
\widehat d_i = \sum_{j \neq i} \frac{1}{\widehat \theta_i + \widehat \theta_j}, \ \  \text{for } i = 1,\dots,n.
\end{equation}
It can be shown that a solution $\widehat \theta$ to (\ref{mom_match_eqn}) with $\widehat \theta_i + \widehat \theta_j > 0$, $i \neq j$, is unique almost surely and, more surprisingly, that the estimator $\widehat \theta$ is {\em consistent}. In other words, the degree sequence of one sample from a maximum entropy distribution essentially determines the distribution for large $n$.  We remark that solutions $\widehat \theta$ to equations (\ref{mom_match_eqn}) also exhibit a rich combinatorial structure which has been explored (within a more general framework) by Sanyal, Sturmfels, and Vinzant using matroid theory and algebraic geometry \cite{SturmEntDisc}.

To see how Theorem \ref{inv_F_cor} applies to this context, consider
the function $F$ determined as in (\ref{F_map}) by setting $g_{ij}(z)
= -1/z$ for $i \neq j$ and $g_{ii}=0$ otherwise. In this case,
equations (\ref{mom_match_eqn}) are concisely expressed
as~$\widehat{d} = F(-\widehat \theta)$.  Using Theorem \ref{inv_F_cor}
and arguments from large deviation theory, one can show that given any
$k > 1$ and for sufficiently large $n$, we have the estimate:
\begin{equation}\label{Eq:AppMLE-1}
|\theta - \widehat \theta|_\infty \leq  \frac{150\sqrt{m}}{\ell} \sqrt{\frac{k \log n}{n}} \ \  \text{ with probability } \geq 1-\frac{3}{n^{k-1}}.
\end{equation}
Thus, $|\theta-\widehat \theta|_\infty \to 0$ in probability
as $n \to \infty$.  We note that analogous arguments with the
$O(\frac{1}{\sqrt{n}})$ bound from~\eqref{Eq:SpectralInftyBound} fail to show
consistency.  We refer the reader to~\cite{diaconis2011} and~\cite{HilWib} for more details on these results.

\section{Reduction to Exact Limiting Cases}
\label{Sec:CornerCase}

To prove Theorem~\ref{Thm:Main}, we need to show that the maximum of
$\|J^{-1}\|_\infty$ over the space of symmetric diagonally dominant
matrices $J \geq S:=\alpha \eye_n + \ell \ones_n\ones_n^\top$ is achieved at $J = S$. A priori, it is not even clear that a maximum exists since this space is not compact. In this section we consider maximizing $\|J^{-1}\|_\infty$ over compact sets of symmetric diagonally dominant matrices $J$, and we show that the maxima occur at the corners of the space. In subsequent sections we analyze these corner matrices in more detail.

Fix $m \geq 1$, and let $\D = \D_m$ denote the set of $n \times n$
matrices of the form $J = S + (m-\ell)P$ where $P$ is some symmetric diagonally dominant matrix
satisfying
\begin{equation*}
0 \leq P_{ij} \leq 1 \;\text{ for }\; i \neq j
\quad \text{ and } \quad
0 \leq \Delta_i(P) \leq 2 \;\text{ for }\; i = 1,\dots,n.
\end{equation*}
We say that $J \in \D$ is a {\em corner matrix} if
\begin{equation*}
P_{ij} \in \{0,1\} \;\text{ for }\; i \neq j
\quad \text{ and } \quad
\Delta_i(P) \in \{0,2\} \;\text{ for }\; i = 1,\dots,n.
\end{equation*}
Equivalently, $J$ is a corner matrix if $P$ is a signless Laplacian
matrix. Let $\T$ denote the set of matrices $J \in \D$ that maximize
$\|J^{-1}\|_\infty$. This set is closed and nonempty since the
function $J \mapsto \|J^{-1}\|_\infty$ is continuous and $\D$ is
compact in the usual topology. Let $\basis_1,\dots,\basis_n$ be the
standard column basis for $\R^n$, and set $\Basis_{ii} =
\basis_i\basis_i^\top$ and $\Basis_{ij} =
(\basis_i+\basis_j)(\basis_i+\basis_j)^\top$ for all $i \neq j$. Our main result in this section is the following.

\begin{proposition}\label{Prop:CornerCaseConnected}
Every $J \in \T$ is path-connected to a corner matrix.
\end{proposition}
\begin{proof}
Let $J \in \mathcal{T}$. We will show that if $\ell <
J_{ij} < m$ for some $i \neq j$, there is a path in $\T$ from $J$ to a matrix $J'$ that differs from $J$ only
in the $(i,j)$-entry, with $J'_{ij} \in \{\ell,m\}$. Similarly, if $0 < \Delta_i(J) < 2(m-\ell)$ for some $1 \leq
i \leq n$, then we find a suitable $J'$ with $\Delta_i(J') \in \{0,2(m-\ell)\}$ that differs from $J$ in the $(i,i)$-entry. Repeatedly applying these steps, it follows that there is a path in $\T$  from any $J \in T$ to a corner matrix.

For the first part, suppose that $\ell < J_{ij} < m$ for some $i \neq j$. Consider the nonempty, closed set
\begin{equation*}
\mathcal{W} = \{J + t\Basis_{ij} \colon \ell \leq J_{ij} + t \leq m \} \;\cap\; \mathcal{T}.
\end{equation*}
We claim that $\mathcal{W}$ contains a matrix $J'$ with $J'_{ij} \in
\{\ell,m\}$. Suppose not, and let $J' \in \mathcal{W}$  be a matrix with
minimum $(i,j)$-entry. By Proposition~\ref{Prop:Deformation} {\bf (a)}
below, $J' + t\Basis_{ij} \in \T$ for all $t$ in a small neighborhood
of the origin. Thus, there is another matrix in $\T$ that has a smaller
$(i,j)$-entry than $J'$, a contradiction. The proof
for the other part is similar (using
Proposition~\ref{Prop:Deformation} {\bf (b)}). 
\end{proof}

To complete the proof of Proposition~\ref{Prop:CornerCaseConnected}, it remains to show the following.

\begin{lemma}\label{Prop:Deformation}
Let $J \in \T$ and $i \neq j$ be distinct indices in $\{1, \ldots, n\}$.
\begin{enumerate}[{\bf (a)}]
  \item If $1 < J_{ij} < m$,  then $J + t\Basis_{ij} \in \T$ for all
    $t \in \R$ in some neighborhood of the origin.
  \item If $0 < \Delta_i(J) < 2(m-1)$, then  $J + t\Basis_{ii} \in \T$
    for all
    $t \in \R$ in some neighborhood of the origin.
\end{enumerate}
\end{lemma}
\begin{proof}
We only prove {\bf (a)} as {\bf (b)} is analogous. Suppose that $1 < J_{ij} < m$ for some $i \neq j$. Let $K = J^{-1}$ and set $K_1,\dots,K_n$ to be its columns. Also, let $q \in \arg\max_{1 \leq p \leq n} |K_p|_1$~and $\mathcal{I} = \{1 \leq p \leq n \colon K_{pq} \neq 0\}$ so that
\begin{equation*}
\sum_{p \in \mathcal{I}} |K_{pq}| = |K_q|_1 = \|K\|_\infty = \|J^{-1}\|_\infty.
\end{equation*}
By the Sherman-Morrison-Woodbury formula, we have
\begin{equation*}
(J + t\Basis_{ij})^{-1} = K - \frac{t}{1+t(\basis_i+\basis_j)^\top K(\basis_i+\basis_j)} \; K\Basis_{ij}K = K - \frac{t}{1+\eta_{ij}t} \; (K_i+K_j)(K_i+K_j)^\top
\end{equation*}
where $\eta_{ij} := (\basis_i+\basis_j)^\top K(\basis_i+\basis_j) > 0$
since $K \succ 0$. The formula implies that for sufficiently small
$\varepsilon > 0$ and all $|t| \leq \varepsilon$, the $(p,q)$-entry of $(J + t\Basis_{ij})^{-1}$ has the same
sign as $K_{pq} \neq 0$ for all $p \in \mathcal{I}$. Let us further
suppose that $\varepsilon$ is small enough so that $J+t\Basis_{ij} \in
\D$ and $1+\eta_{ij}t > 0$ for all $|t|\leq \varepsilon$. The $1$-norm
of the $q$-th column $(J + t\Basis_{ij})_q^{-1}$ of the matrix of $(J
+ t\Basis_{ij})^{-1}$ now satisfies:
\begin{equation*}
|(J + t\Basis_{ij})_q^{-1}|_1 = \|J^{-1}\|_\infty -
\frac{\phi_{ij}t}{1+\eta_{ij}t} + \frac{\psi_{ij}|t|}{1+\eta_{ij}t},
\quad \text{for all $|t|\leq \varepsilon$}
\end{equation*}
where $\phi_{ij} := (K_{iq}+K_{jq}) \sum_{p \in \mathcal{I}} \sign(K_{pq})(K_{ip}+K_{jp})$ and
$\psi_{ij} := |K_{iq}+K_{jq}| \sum_{p \notin \mathcal{I}} |K_{ip}+K_{jp}|$. Recalling that $J \in \T$ achieves the maximum $\|J^{-1}\|_\infty$, it follows that 
\begin{equation*}
\phi_{ij}t \geq \psi_{ij}|t|, \ \  \text{for all } |t| < \varepsilon.
\end{equation*}
This inequality implies that $\psi_{ij} = \phi_{ij}= 0$, so
$\|(J+t\Basis_{ij})^{-1}\|_\infty = \|J^{-1}\|_\infty$ for all $|t| <
\varepsilon$ as required.
\end{proof}

\section{Boundary Combinatorics and Exact Formulae}
\label{Sec:Infinity}

In the previous section, we saw that corner matrices played an important
role in the optimization of $\|J^{-1}\|_\infty$ over all matrices $J
\geq S :=  \alpha \eye_n + \ell \ones_n\ones_n^\top$. A corner matrix may be expressed as:
$$J = S+ (m-\ell)P,$$
in which $P$ is a symmetric diagonally dominant matrix with $P_{ij} \in
\{0,1\}$ for all $i \neq j$ and $\Delta_i(P) \in \{0,2\}$ for $1 \leq
i \leq n$. Every such matrix $P$ is a \textit{signless Laplacian} of
an undirected unweighted graph (possibly with self-loops) $G = (V,E)$
on vertices $V = \{1,\dots,n\}$ and edges $E$. That is, we can write
$P = D + A$ where $D$ is the diagonal degree matrix of $G$ and $A$ is
its adjacency matrix. We study the limit (\ref{eqn:Nlim}) using the
combinatorics of the graphs associated to the matrix $P$.

\begin{example} Let $S = (n-2)\eye_n +  \ones_n\ones_n^\top$
and $G$ be the chain graph from Example~\ref{Ex:Graphs}. 
For $n = 4$, 
\begin{equation*}
S + tP = 
\begin{pmatrix}
3+2t & 1+t & 1 & 1+t \\
1+t & 3+2t & 1+t & 1 \\
1 & 1+t & 3+2t & 1+t \\
1+t & 1 & 1+t & 3+2t
\end{pmatrix},
\end{equation*}
with inverse:
\begin{equation*}
(S + tP)^{-1} = 
\frac{1}{4(2t+3)(t+1)}
\begin{pmatrix}
t^2+5t+5 & -(t+1)^2 & t^2+t-1 & -(t+1)^2 \\
-(t+1)^2 & t^2+5t+5 & -(t+1)^2 & t^2+t-1 \\
t^2+t-1 & -(t+1)^2 & t^2+5t+5 & -(t+1)^2 \\
-(t+1)^2 & t^2+t-1 & -(t+1)^2 & t^2+5t+5
\end{pmatrix}.
\end{equation*}
Each entry of $(S+tP)^{-1}$ is a rational function of $t$, with numerator and denominator both quadratic. Thus, each entry converges to a constant as $t \to \infty$, and from the expression above, we see that the limit matrix $N$ has entries $N_{ij} = (-1)^{i+j}/8$, as predicted by the formula in Example~\ref{Ex:Graphs}.

If one adds the edge $\{1,3\}$, the corresponding matrix $(S + tP)^{-1}$ is
\begin{equation*}
(S + tP)^{-1} = \frac{1}{4(t+3)(t+1)}
\begin{pmatrix}
2t+5 & -t-1 & -1 & -t-1 \\
-t-1 & 3t+5 & -t-1 & t-1 \\ 
-1   & -t-1 & 2t+5 & -t-1 \\
-t-1 & t-1  & -t-1 & 3t+5
\end{pmatrix}.
\end{equation*}
Thus $N = 0$, as the graph is no longer bipartite (see Corollary \ref{Thm:blockconstants} below).
\qed
\end{example}

We begin with the following simple fact that allows us to invert certain classes of matrices explicitly.
 
\begin{lemma} \label{Thm:invertsimplematrix}
The following identity holds for any $\alpha \neq 0$ and $\ell \neq -\alpha/n$:
$$
(\alpha \eye_n + \ell \ones_n\ones_n^\top )^{-1} = \frac{1}{\alpha} \eye_n -
\frac{\ell}{\alpha(\alpha+\ell n)} \ones_n\ones_n^\top.
$$
\end{lemma}

Given a signless Laplacian $P$, let $L \in \R^{n\times|E|}$ be the incidence matrix of the graph $G$ associated to $P$; that is, for every vertex $v \in V$ and edge $e \in E$, we have:
\begin{equation*}
L_{v,e} = 
\begin{cases}
1 \quad & \text{ if $v$ is in $e$ and $e$ is not a self-loop}, \\
\sqrt{2} & \text{ if $v$ is in $e$ and $e$ is the self-loop $(v,v)$}, \\
0 & \text{ otherwise.} 
\end{cases}
\end{equation*}
Consequently, $P = LL^\top$. Using this decomposition of $P$, we
derive the following formula for $N$.

\begin{proposition}\label{Prop:Pseudoinverse}
The limit $N$ in (\ref{eqn:Nlim}) satisfies:
\begin{equation*}
N = S^{-1} - S^{-1/2} (X^\top)^\dagger X^\top S^{-1/2},
\end{equation*}
where $X = S^{-1/2} L$ and $(X^\top)^\dagger$ is the Moore-Penrose
pseudoinverse of $X^\top$. Furthermore, $NL=0$.
\end{proposition}
\begin{proof}
Using the Sherman-Morrison-Woodbury matrix formula to expand $(S + tLL^\top)^{-1}$, we calculate:
\begin{equation*}
\begin{split}
N &=\textstyle \lim_{t \to \infty} \,(S + tLL^\top)^{-1} \\
&= \textstyle \lim_{t \to \infty} \,[ \,S^{-1} - S^{-1}L(t^{-1} \eye + L^\top S^{-1}L)^{-1}L^\top S^{-1}\, ] \\
&= \textstyle S^{-1} - S^{-1/2} \, [ \, \lim_{t \to \infty} X(t^{-1} \eye + X^\top X)^{-1} \, ] \, X^\top S^{-1/2} \\
&= S^{-1} - S^{-1/2} (X^\top)^\dagger X^\top S^{-1/2},
\end{split}
\end{equation*}
where in the last step we used an elementary pseudoinverse identity in
matrix analysis~\cite[p.\ 422, Ex.\ 9]{HornJohnson}. To show that
$NL=0$, we note that $N$ is symmetric and that
\begin{equation*}
\begin{split}
(NL)^\top = L^\top N &= X^\top S^{-1/2} - X^\top (X^\top)^\dagger
X^\top S^{-1/2} = 0. \qedhere
\end{split}
\end{equation*}
\end{proof}

Let $N_1, \ldots, N_n$ denote the columns of the matrix $N$.  The following is immediate from $NL=0$.
\begin{corollary} \label{Thm:alternating columns}
For each edge $\{i,j\}$ of $G$, we have $N_i = -N_j$. In particular,
$N_i=0$ for each self-loop $\{i,i\}$.
\end{corollary}

 Suppose that the graph $G$ has connected components $G_1, \ldots, G_k$. After
 relabeling the vertices, both $P$ and $L$ are block-diagonal with
 matrices $P_1, \ldots, P_k$ and $L_1, \ldots, L_k$ along the diagonal. Furthermore, $P_i = L_iL_i^\top $ for
each $1 \leq i \leq k$. The components of $G$ also induce a block-structure on the limit $N$, and we
denote these blocks by $N[i,j]$ for $i,j \in \{1,\ldots, k\}$. These blocks
display many interesting symmetries. Firstly, the entries in each
block are all equal up to sign. Secondly, the signs in a block $N[i,j]$ depend
on the bipartite structures of the components $G_i$ and $G_j$. We say that a bipartite graph is $(p,q)$-\emph{bipartite} if the partitions are of sizes $p$ and $q$, respectively. Note that bipartite graphs cannot have self-loops.

\begin{corollary} \label{Thm:blockconstants}
In each block $N[i,j]$, the entries are equal up to sign. If
the $i$-th component $G_i$ of $G$ is not bipartite, then $N[i,j] = N[j,i] =0$ for
all $j = 1,\dots,k$.  In particular, $N =0$ if and only if all components
$G_i$ are not bipartite. Suppose $G_i$ is $(p_i,q_i)$-bipartite and $G_j$ is
$(p_j,q_j)$-bipartite. Then, after relabeling the vertices, the matrix $N[i,j]$ has the block structure
$$
N[i,j] \,\,=\,\,c_{ij}\left(\begin{array}{cc} \ones_{p_i}
    \ones_{p_j}^\top  \ &- \ones_{p_i}
    \ones_{q_j}^\top  \vspace{0.05in} \\-\ones_{q_i} \ones_{p_j}^\top &
  \ones_{q_i}\ones_{q_j}^\top \end{array}\right), \ \ \text{for some constant $c_{ij} \in \R$}.
$$
\end{corollary}

Now that we understand the block structure of $N$, we want to compute the constants $c_{ij}$. Our approach is to simplify the formula in Proposition
\ref{Prop:Pseudoinverse} by expressing the incidence matrix $L$ in a
more suitable form. 

\begin{proposition}
Let $G$ be a connected graph with $n$ vertices. If $G$ is not
bipartite, then $\rank L  = n$.
\end{proposition}
\begin{proof}
It suffices to prove that the positive semidefinite matrix $P = LL^\top
$ has rank $n$; i.e., if $x \in \R^n$ satisfies $x^\top Px = 0$,
then $x=0$. Write
\begin{equation*}
x^\top Px = \sum_{i=1}^n \sum_{j=1}^n P_{ij} x_i x_j = \sum_{\substack{i \neq j\\\{i,j\} \in E}} (x_i + x_j)^2 + \sum_{i \colon \{i,i\} \in E} 2x_i^2
\end{equation*}
so that $x^\top Px = 0$ implies
\begin{equation}\label{Eq:EdgeZero}
x_i = 0 \; \text{ if } \{i,i\} \in E \quad \text{ and } \quad x_i + x_j = 0 \; \text{ if } \{i,j\} \in E.
\end{equation}
If $G$ has a self-loop $\{i, i\}$, then $x_i = 0$
by~\eqref{Eq:EdgeZero}. If $G$ has no self-loops, there is an odd
cycle $(i_1, i_2, \dots, i_{2m+1})$ for some $m \geq 1$ (because $G$ is not bipartite). Applying condition~\eqref{Eq:EdgeZero} to each successive edge in this cycle shows $x_{i_1} = \cdots = x_{i_{2m+1}} = 0$. In either case $x_i = 0$ for some vertex $i \in V$. A repeated application of~\eqref{Eq:EdgeZero} then reveals that $x_j = 0$ for all vertices $j$ connected to $i$. Since $G$ is connected, we have $x = 0$ as desired.
\end{proof}

Let $\basis_i \in \R^{n}$ denote the $i$-th standard basis column vector.
\begin{proposition}\label{Prop:L=UA}
Let $G$ be a connected bipartite graph on $n$ vertices. Then
$\rank L = n-1$. Let $U$ be the
$n \times (n-1)$ matrix with columns $\basis_1
+ \sigma_i \basis_i$ for $2\leq i \leq n$, where $\sigma_i=-1$ if
vertex $i$ is in the same partition as vertex $1$ and $\sigma_i=1$
otherwise. Then, $L=UB$ for some $(n-1) \times |E|$ matrix $B$ of rank
$n-1$. 
\end{proposition}
\begin{proof}
Recall that the columns of $L$ are of the form $\basis_i + \basis_j$ where
$\{i,j\} \in E$ (there are no self-loops because $G$ is bipartite). There
is a path $(1, i_1, i_2, \cdots, i_m, j)$ from vertex $1$ to vertex $j$ for
each $2 \leq j \leq n$, so
$$
\basis_1 + (-1)^m\basis_j =  (\basis_1+\basis_{i_1}) - (\basis_{i_1} +
\basis_{i_2}) + \cdots + (-1)^m (\basis_{i_m} +\basis_j).
$$
Thus, $\rank L \geq n-1$. Conversely, for each edge $\{i,j\}$ where vertex $i$ is in the
same partition as vertex $1$, 
$$\basis_i + \basis_j =  - (\basis_1 - \basis_i) + (\basis_1 + \basis_j),$$
so $\rank L \leq n-1$. This equation also allows us to write
$L=UB$ for some matrix $B$, and the rank condition on $B$ follows from that of $L$.
\end{proof}

Recall that $G$ has $k$ components $G_1, \ldots, G_k$ and $L$ is block-diagonal with matrices $L_1, \ldots, L_k$. If $G_i$ is bipartite, we write $L_i =  U_i B_i$ as in Proposition~\ref{Prop:L=UA}. If $G_i$ is not bipartite, we write $L_i = U_i B_i$ where $U_i = \eye$ is the identity matrix and $B_i = L_i$. Let $r$ be the number of components of $G$ which are bipartite. If $U$ and $B$ are block-diagonal matrices constructed from $U_1, \ldots,U_k$ and $B_1, \ldots, B_k$, then $L = UB$ where $U \in \R^{n{\times}(n-r)}$ and $B \in \R^{(n-r){\times}|E|}$ both have rank $(n-r)$. Note that $U$ contains information about
the sizes of the bipartitions of each component whereas $B$ contains information about the edges. Let $U_{(p,q)}$ denote the matrix 
$$
U_{(p,q)} := \left(\begin{array}{cc} \vspace{0.05in} \ones_{p-1}^\top  & \ones_q^\top  \\
\vspace{0.05in}-\eye_{p-1} & 0 \\
0 & \eye_{q}
\end{array} \right).
$$
After relabeling the vertices, we have
\begin{equation}\label{bigUeqn}
U = \left( \begin{array}{ccccc} 
U_{(p_1,q_1)} & 0 & \cdots & 0 & 0\\
0 & U_{(p_2,q_2)} & \cdots & 0 & 0\\
\vdots& \vdots &  & \vdots & \vdots\\
0 & 0 & \cdots & U_{(p_r,q_r)} & 0\\
0 & 0 & \cdots & 0 & \eye_s
\end{array} \right) = 
\left( \begin{array}{cc} 
\widetilde{U} & 0 \\
0 & \eye_s
\end{array} \right),
\end{equation}
where $s =  n - \sum_{i=1}^r (p_i+q_i)$ is the total number of
vertices in the non-bipartite components of $G$. Our next result shows that dependence on the matrix $B$ can be removed in Proposition
\ref{Prop:Pseudoinverse}.  This new formula for $N$ also gives us a method to compute its entries explicitly. 

\begin{proposition} \label{Thm:newNformula}
Let $U$ be as in (\ref{bigUeqn}).  The limit $N$ in (\ref{eqn:Nlim}) satisfies:
$$N =  S^{-1} - S^{-1}U(U^\top S^{-1}U)^{-1}U^\top  S^{-1}. $$
Therefore, $N$ depends only on the sizes of the bipartitions of each component $G_i$.
\end{proposition}
\begin{proof}
Write $L = UB$ and $P = U(BB^\top U^\top )$. First, note that $BB^\top $ is positive definite
since $B \in \R^{(n-r){\times}|E|}$ has rank $(n-r)$. Now, $U^\top S^{-1}U$ is positive semidefinite, being a congruence of a positive definite matrix $S^{-1}$ (Sylvester's Law of Inertia). Moreover, since $U$ is injective as a linear map and $S^{-1} \succ 0$, for any $x \in \R^{n-r}$:
$$x^\top U^\top S^{-1}Ux = 0 \quad \Rightarrow \quad Ux = 0 \quad \Rightarrow
\quad x = 0.$$
Thus, $U^\top S^{-1}U$ is positive definite. Then the
eigenvalues of $BB^\top U^\top S^{-1}U$ are all positive, being a product of two positive definite matrices (see \cite[Lemma 2]{HJ02}). Therefore, the matrix $t^{-1} \eye + BB^\top U^\top S^{-1}U$ is invertible for all $t > 0$, and so by the Sherman-Morrison-Woodbury formula, we have
\begin{equation*}
\begin{split}
(S + tP)^{-1} &= (S + tU(BB^\top U^\top))^{-1} \\
&= S^{-1} - S^{-1}U(t^{-1}\eye + BB^\top U^\top S^{-1}U)^{-1}BB^\top U^\top S^{-1}.
\end{split}
\end{equation*}
Taking limits of this equation as $t \rightarrow \infty$, the result follows.
\end{proof}

To state our ultimate formula for $N$, let us first define:
$$
\gamma = \sum_{i=1}^r \frac{(p_i-q_i)^2}{p_i+q_i},
$$
$$
y^\top = \left( \frac{p_1-q_1}{p_1+q_1} (\ones_{p_1}^\top ,
  -\ones_{q_1}^\top ), \; \ldots, \; \frac{p_r-q_r}{p_r+q_r} (\ones_{p_r}^\top ,
  -\ones_{q_r}^\top ), \; 0 \cdot \ones_s^\top  \right),
$$
\medskip
$$
Y= \left(\begin{array}{cccc}
\displaystyle\frac{1}{p_1+q_1} \left(\begin{array}{cc}
\vspace{0.05in} \ones_{p_1} \ones_{p_1}^\top  & -\ones_{p_1}\ones_{q_1}^\top \\
- \ones_{q_1} \ones_{p_1}^\top  & \ones_{q_1}\ones_{q_1}^\top 
\end{array}\right) & \cdots & 0 & 0 \\
\vdots &  &\vdots & \vdots \\
0 & \cdots & \displaystyle\frac{1}{p_r+q_r} \left(\begin{array}{cc}
\vspace{0.05in} \ones_{p_r} \ones_{p_r}^\top  & -\ones_{p_r}\ones_{q_r}^\top \\
 -\ones_{q_r} \ones_{p_r}^\top  & \ones_{q_r}\ones_{q_r}^\top 
\end{array}\right) & 0 \\
0 & \cdots & 0 & 0 \cdot \eye_s
\end{array}\right).
\medskip
$$
\begin{proposition}
Set $\gamma$, $y$, and $Y$ as above, which depend only on the bipartite
structures of the components of the underlying graph $G$. We have the following formula for the limit in (\ref{eqn:Nlim}):
$$
N = \lim_{t \rightarrow \infty} (S+tP)^{-1} = \frac{1}{\alpha}Y - \frac{\ell}{\alpha(\alpha+\ell \gamma)}yy^\top.
$$
\end{proposition}
\begin{proof} We outline the computation of $N$. For simplicity, let us write $S^{-1} = a\eye_n - b\ones_{n} \ones_{n}^\top $. Then, 
\begin{align}\label{Eq:UUSUU}
\nonumber U(U^\top S^{-1}U)^{-1}U^\top  &= U(U^\top (a\eye_n -
b\ones_{n}\ones_{n}^\top )U)^{-1}U^\top  \\
&= U(aW- bvv^\top  )^{-1} U^\top,
\end{align}
where $W = U^\top U$ and $v = U^\top \ones_n$. By the Sherman-Morrison-Woodbury identity, we have that
\begin{align*}
(aW- bvv^\top  )^{-1} &= a^{-1}W^{-1} -
\frac{a^{-1}W^{-1}(-bv)v^\top a^{-1}W^{-1}}{1+v^\top a^{-1}W^{-1}(-bv)}\\
&=\frac{1}{a}W^{-1} + \frac{b}{a(a-b\varsigma)}W^{-1}vv^\top W^{-1},
\end{align*}
where $\varsigma = v^\top W^{-1}v$ (it is easy to show that $a-b\varsigma > 0$). Substituting this back into
\eqref{Eq:UUSUU} gives us
\begin{align} \label{Eq:UUSUUsimplified}
U(U^\top S^{-1}U)^{-1}U^\top  &= \frac{1}{a} Z +\frac{b}{a(a-b\varsigma)} zz^\top,
\end{align}
where $Z = U(U^\top U)^{-1} U^\top$, $z =  Z\ones_n$, and $\varsigma =
\ones_n^\top  Z\ones_n$. From the block-diagonal structure of $U$, the matrix $Z$ is also block-diagonal
with blocks 
\begin{align*}
U_i(U_i^\top U_i)^{-1}U_i^\top  &=\eye_{p_i+q_i} - \frac{1}{p_i+q_i}
\begin{pmatrix}
\vspace{0.05in} \ones_{p_i} \ones_{p_i}^\top  & -\ones_{p_i}\ones_{q_i}^\top \\
- \ones_{q_i} \ones_{p_i}^\top  & \ones_{q_i}\ones_{q_i}^\top 
\end{pmatrix},
\ \ \text{ for } i = 1,\dots,r,
\end{align*}
where we have computed $(U_i^\top U_i)^{-1}$ using Lemma
\ref{Thm:invertsimplematrix}. Next, one shows that $Z = \eye_n - Y$, $z = \ones_n-y$, and
$\varsigma = n-\gamma$. Finally, substituting
\begin{align}\label{eq:formulaforSinverse}
  S^{-1} = a\eye_n - b\ones_{n}\ones_{n}^\top, \quad a =
  \frac{1}{\alpha}, \quad b = \frac{\ell}{\alpha(\alpha + \ell n)},  
\end{align}
and equation \eqref{Eq:UUSUUsimplified} into Proposition \ref{Thm:newNformula} gives the desired result.
\end{proof}

\begin{corollary}
For $1 \leq i,j \leq r$, the constants $c_{ij}$ in the blocks $N[i,j]$ in Corollary \ref{Thm:blockconstants} are:
$$
c_{ii} = \frac{\ell}{\alpha(\alpha+\ell\gamma)(p_i+q_i)}
\left(\frac{\alpha}{\ell}+\gamma-\frac{(p_i-q_i)^2}{p_i+q_i}\right),
$$
$$
c_{ij} =
\frac{-\ell}{\alpha(\alpha+\ell\gamma)}\left(\frac{p_i-q_i}{p_i+q_i}\right)
\left(\frac{p_j-q_j}{p_j+q_j}\right), \quad j \neq i.
$$
\end{corollary}
\medskip
Finally, we write down an explicit formula for $\| N \|_\infty$ and verify that $\| N \|_\infty
\leq \| S^{-1} \|_\infty$.

\begin{corollary}\label{Thm:formulaforNinfinity} If $r=0$, then
  $\|N\|_\infty=0$. If $r \geq 1$, let $d =  \sum_{i=1}^r |p_i-q_i|$. Then
$$
\| N \|_\infty \,\,=\,\, \frac{1}{\alpha} +
\frac{\ell}{\alpha(\alpha + \ell\gamma)}  \max_{1\leq i \leq r}\frac{|p_i-q_i|(d-2|p_i-q_i|)}{p_i+q_i} .
$$
\end{corollary}
\begin{proof} Recall that if $r = 0$ (i.e.\ no component of $G$ is
  bipartite), then $N = 0$ by Corollary~\ref{Thm:blockconstants}. Now
  if $r \geq 1$, we may assume that $p_i \geq q_i$ for
  all $i$ after a relabeling of vertices. Observe that $c_{ii} > 0$ and $c_{ij} \leq 0$ for
  all $i$ and $j \neq i$. Indeed, this follows from $\alpha > 0$,
  $\ell > 0$, and 
$$
\gamma = \sum_{i=1}^r
    \frac{(p_j-q_j)^2}{p_j+q_j} \geq 0, \quad 
    \gamma-\frac{(p_i-q_i)^2}{p_i+q_i} = \sum_{j\neq i} \frac{(p_j-q_j)^2}{p_j+q_j} \geq 0.
$$
Consequently, the $1$-norm of rows in the
$i$-th block of $N$ is
\begin{align}\label{Eq:OneNormRow}
(p_i+q_i)c_{ii} - \sum_{j\neq i} (p_j+q_j) c_{ij} &\,\,=\,\, \frac{1}{\alpha} +
\frac{\ell (p_i-q_i)(d-2p_i+2q_i)}{\alpha(\alpha +  \ell\gamma)(p_i+q_i)},
\end{align}
and $\| N \|_\infty $ is the maximum of these $1$-norms.
\end{proof}

\begin{corollary}\label{Thm:limitatinfinity}  For all signless Laplacians $P$, we have
$$
\| N \|_{\infty} = \| \lim_{t \rightarrow \infty} (S+tP)^{-1} \|_\infty \leq \| S^{-1} \|_\infty = \frac{1}{\alpha} +
\frac{\ell(n-2)}{\alpha(\alpha +  \ell n)} ,
$$
with equality if and only if $P$ is the zero matrix.
\end{corollary}
\begin{proof}
Let $N = \lim_{t \rightarrow \infty} (S+tP)^{-1}$. If $r = 0$ then $\|N\|_\infty = 0$, so suppose that $r \geq 1$. As before, assume that $p_i \geq q_i$ for all $i = 1,\dots,r$. It suffices
to show that the $1$-norms of the rows of $N$, as computed 
in \eqref{Eq:OneNormRow}, are at most $ \| S^{-1} \|_\infty $, with
equality achieved only at $J = S$. The inequality is trivial if
$p_i=q_i$ so we may assume $p_i-q_i \geq 1$. We outline the proof and
leave the details to the reader. The key is to show that
\begin{align*}
  (\alpha+\ell n)(d-2p_i+2q_i) \left(\frac{p_i-q_i}{p_i+q_i} \right)
  \,\, \leq \,\,(\alpha+\ell n)( d - 2) \,\, \leq \,\,(\alpha+\ell \gamma)( n- 2).
\end{align*}
The latter inequality is equivalent to
\begin{align*}
  0 \,\, \leq \,\, (n-d)(\alpha + 2\ell- \ell n) + 2\ell(d-\gamma)+ n \ell (n-2d+\gamma).
\end{align*}
The first summand is nonnegative because $S$ is diagonally
dominant, while the last summand satisfies
\begin{align*}
  n-2d+\gamma\,\,= \,\, s + \sum_{j=1}^r
  \frac{4q_j^2}{p_j+q_j} \,\,\geq \,\,0.
\end{align*}
Finally, if equality $\| N \|_\infty=\| S^{-1} \|_\infty$ is achieved, then $s
= 0$ and $q_j = 0$ for all $j$, so $P = 0$.
\end{proof}

\section{Analysis of $\|J^{-1}\|_\infty$ in a Neighborhood of $S$}
\label{Sec:SmallT}

The arguments in Sections~\ref{Sec:CornerCase} and~\ref{Sec:Infinity}
show that for $J \geq  S:= \alpha \eye_n + \ell \ones_n\ones_n^\top$,
the maximum of $\|J^{-1}\|_\infty$ is attained at $J =S$. To prove
that $S$ is the unique maximizer, we will show that
$\|J^{-1}\|_\infty$ is strictly decreasing near $S$. Let $P \geq 0$ be a nonzero symmetric diagonally dominant matrix, and consider the function
\begin{equation*}
f(t) = \|(S + tP)^{-1}\|_\infty, \quad t \geq 0.
\end{equation*}
In our proof, we study the linear part $S^{-1} - t S^{-1} P
S^{-1}$ of the Neumann series for $(S + tP)^{-1}$. Let us define 
$g(t) = \|S^{-1} - t S^{-1} P S^{-1} \|_\infty$ and $h(t) = f(t)-g(t)$.
Our main result in this section is the following.

\begin{proposition}\label{Lem:Decreasing}
The function $f(t)$ is differentiable at $t = 0$ and $f'(0) < 0$.
\end{proposition}
\begin{proof}
Since $f(t) = g(t)+h(t)$, the result follows from Propositions~\ref{Prop:h(t)} and \ref{Prop:g(t)}.
\end{proof}

\begin{proposition}\label{Prop:h(t)}
The function $h(t)$ is differentiable at $t = 0$ and $h'(0) = 0$ .
\end{proposition}
\begin{proof}
For sufficiently small $t > 0$, by the Neumann series for $(\eye + t S^{-1/2} P S^{-1/2} )^{-1}$ we can write
\begin{equation*}
(S + tP)^{-1} = S^{-1/2} (\eye + t S^{-1/2} P S^{-1/2} )^{-1} S^{-1/2} = S^{-1/2} \left( \sum_{k = 0}^\infty (-t)^k \big(S^{-1/2} P S^{-1/2} \big)^k \right) S^{-1/2}.
\end{equation*}
By the reverse triangle inequality and submultiplicativity of $\| \cdot \|_\infty$, we have
\begin{equation*}
\begin{split}
|h(t)| &\leq \left\| (S + tP)^{-1} - S^{-1} + t S^{-1} P S^{-1} \right\|_\infty \\
&= \left\| S^{-1/2} \left( \sum_{k = 2}^\infty (-t)^k \big(S^{-1/2} P S^{-1/2} \big)^k \right) S^{-1/2} \right\|_\infty \\
&= t^2 \: \left\| S^{-1/2} \: \big(S^{-1/2}PS^{-1/2}\big)^2 \: S^{1/2} \: (S + tP)^{-1} \right\|_\infty \\
&\leq 2t^2 \: \|S^{-1}P S^{-1}P\|_\infty \: \|S^{-1}\|_\infty,
\end{split}
\end{equation*}
where the last inequality holds for sufficiently small $t > 0$ since
by continuity $\|(S + tP)^{-1}\|_\infty \leq 2\|S^{-1}\|_\infty$ for small $t$. It follows that $h'(0) = \lim_{t \to 0} (h(t)-h(0))/t = 0$.
\end{proof}

\begin{proposition}\label{Prop:g(t)}
The function  $g(t)$ is differentiable at $t = 0$ and $g'(0) < 0$.
\end{proposition}
\begin{proof}
Set $Q = S^{-1} P S^{-1}$. Note that for sufficiently small $t > 0$,
the entries of $S^{-1} - tQ$ have the same sign as the corresponding
entries of $S^{-1}$. Since $(S^{-1})_{ii} > 0$ and $(S^{-1})_{ij} < 0$ for $i \neq j$, we can write 
\begin{equation*}
\begin{split}
g(t) = \|S^{-1} - t Q \|_\infty
&= \textstyle \max_i \,\,\sum_j \big| (S^{-1})_{ij} - tQ_{ij} \big| \\
&= \textstyle \max_i \,\,\big[ \,(S^{-1})_{ii} - tQ_{ii} +
\sum_{j\neq i} \big( tQ_{ij} - (S^{-1})_{ij} \big) \,\big] \\
&= \textstyle \max_i \,\,\big[ -(Q_{ii}-\sum_{j\neq i} Q_{ij})\,t +
\|S^{-1}\|_\infty \big] \\
&= -\xi t + \|S^{-1}\|_\infty.
\end{split}
\end{equation*}
where $\xi= \min_i \,(Q_{ii}-\sum_{j\neq i} Q_{ij})> 0$ by
Proposition~\ref{Prop:PosDiag} below. Thus, $g'(0) = -\xi < 0$ as required.
\end{proof}

\begin{proposition}\label{Prop:PosDiag}
Let $Q = S^{-1} P S^{-1}$. Then $Q_{ii} - \sum_{j\neq i} Q_{ij} > 0$
for all $i$.
\end{proposition}
\begin{proof}
For simplicity, let us write $S^{-1} = a\eye_n - b\ones_{n}
\ones_{n}^\top $. Then,
\begin{equation*}
\begin{split}
Q = (a\eye - b\ones \ones^\top) \: P \: (a\eye - b\ones \ones^\top)
= a^2P - abp \ones^\top - ab\ones p^\top + b^2 \pi\ones \ones^\top 
\end{split}
\end{equation*}
where $p = P\ones = (p_1, \ldots, p_n)$ and $\pi = \ones^\top P \ones$. It is
straightforward to check that
\begin{align*}
  Q_{ii}-\textstyle \sum_{j\neq i} Q_{ij} \,=\, a^2(2P_{ii}-p_i) + abp_i (n-4) + b\pi (a+2b-bn).
\end{align*}
From equation \eqref{eq:formulaforSinverse}, we get $a/b =
\alpha/\ell + n$. Substituting this relation and
rearranging gives us
\begin{align*}
  Q_{ii}-{\textstyle \sum_{j\neq i} Q_{ij} } \,=\,\,
  &b^2\left[\left(\frac{\alpha}{\ell}+n-2\right)\left(\frac{\alpha}{\ell}+4\right)+4\right]\Delta_i(P) \,\,\,+\,\,\, \\
 &b^2\left[2\left(\frac{\alpha}{\ell}+n-1\right)\Big(n-3\Big) +
   \frac{\alpha}{\ell}\right] P_{ii}
  \,\,+\,\, b^2\left(\frac{\alpha}{\ell}+2\right) \sum_{j, k\neq i}
  P_{jk}.
\end{align*}
Because $S$ is diagonally dominant, we have $\alpha/\ell \geq
n-2 > 0$. It is not difficult to deduce that if $P \neq 0$, then the above expression is
always positive, as required.
\end{proof}

\section{Proof of Theorem~\ref{Thm:Main}}
\label{Sec:MainProof}

Theorem~\ref{Thm:Main} is a special case of the following theorem when
$\alpha = (n-2)\ell$.

\begin{theorem}\label{Thm:GeneralMain}
Let $n \geq 3$ and suppose $S = \alpha \eye_n + \ell \ones_n
\ones_n^\top$ is diagonally dominant with $\alpha,\ell > 0$. For all
$n \times n$ symmetric diagonally dominant matrices $J \geq S$, we have
\begin{equation*}
\|J^{-1}\|_\infty \leq \|S^{-1}\|_\infty = \frac{\alpha+2\ell(n-1)}{\alpha(\alpha+\ell n)}.
\end{equation*}
Furthermore, equality is achieved if and only if $J =
S$.
\end{theorem}
\begin{proof} Recall from Section~\ref{Sec:CornerCase} that $\D_m$ is the
set of symmetric diagonally dominant matrices $J$ with $\ell \leq J_{ij} \leq m$
and $\Delta_i(S) \leq \Delta_i(J) \leq \Delta_i(S) + 2(m-\ell)$ for all $i \neq j$. Recall also that $
S+(m-\ell)P \in \D_m$ is a corner matrix
if $P$ is a signless
Laplacian. Let $\T_m$ be the set of matrices $J \in \D_m$ maximizing
$\|J^{-1}\|_\infty$. We claim that for sufficiently large $m > \ell$, we have
$\T_m = \{S\}$. Indeed, from Corollary~\ref{Thm:limitatinfinity} and for large $m$:
\begin{equation*}
\|J^{-1}\|_\infty < \|S^{-1}\|_\infty, \ \  \text{for all corner
  matrices }J \in \D_m \setminus \{S\}.
\end{equation*}
Thus, by Proposition~\ref{Prop:CornerCaseConnected}, every $J \in \T_m$ must be
path-connected to the corner matrix $S$. Since
Proposition~\ref{Lem:Decreasing} implies that $S$ is
an isolated point in $\T_m$, we must have $\T_m = \{S\}$ as claimed.

Finally, suppose $J^\ast \geq S$ is a symmetric diagonally dominant
matrix with $\|(J^\ast)^{-1}\|_\infty \geq
\|S^{-1}\|_\infty$. We will show that $J^\ast= S$, which proves
Theorem~\ref{Thm:Main}. We assume that $m$ is sufficiently large
with $J^\ast \in \D_m$ and $\T_m = \{S\}$. Then $S$ is the unique
maximizer of $\|J^{-1}\|_\infty$ for $J \in \D_m$ so that $J^\ast = S$, as desired. 
\end{proof}

\section{Extensions of Hadamard's Inequality}
\label{Sec:Hadamard}

Our arguments for proving Theorem \ref{Thm:DetBound} are inspired by block LU factorization ideas in \cite{demmel1992block}. For $1 \leq i \leq n$, let $J_{(i)}$ be the lower right $(n-i+1){\times}(n-i+1)$ block of $J$, so $J_{(1)} = J$ and $J_{(n)} = (J_{nn})$.  Also, for $1 \leq i \leq n-1$, let $b_{(i)} \in \R^{n-i}$ be the column vector such that
\begin{equation*}
J_{(i)} = 
\begin{pmatrix} J_{ii} & b_{(i)}^{\top} \\ b_{(i)} & J_{(i+1)} \end{pmatrix}.
\end{equation*}
Then our block decomposition takes the form, for $1 \leq i \leq n-1$,
\begin{equation*}
J_{(i)} = 
\begin{pmatrix} 1 & U_{(i)} \\ 0 & \eye_{n-i} \end{pmatrix}
\begin{pmatrix} s_i & 0  \\ b_{(i)} & J_{(i+1)} \end{pmatrix}
\end{equation*}
with
\begin{equation*}
s_i = J_{ii} \left(1 - \frac{b_{(i)}^{\top}J_{(i+1)}^{\,-1}b_{(i)}}{J_{ii}} \right)
\quad\text{ and }\quad
U_{(i)} = b_{(i)}^{\top}J_{(i+1)}^{\,-1}.
\end{equation*}
Notice that 
$\det(J) =  J_{nn} \prod_{i=1}^{n-1} s_i$, or equivalently,
\begin{equation}\label{Eq:HadamardProdInit}
\frac{\det(J)}{\prod_{i=1}^n J_{ii}} = \prod_{i=1}^{n-1} \frac{s_i}{J_{ii}} = \prod_{i=1}^{n-1} \left(1 - \frac{b_{(i)}^{\top}J_{(i+1)}^{\,-1}b_{(i)}}{J_{ii}} \right).
\end{equation}
It remains to bound each factor $s_{i}/J_{ii}$. We first establish the following results.

\begin{lemma}\label{Lem:EigBalanced}
Let $J$ be a symmetric diagonally balanced $n \times n$ matrix with $0 < \ell \leq J_{ij} \leq m$ for $i \neq j$. Then $\ell S \preceq J \preceq mS$, and the eigenvalues $\lambda_1 \leq \dots \leq \lambda_n$ of $J$ satisfy
\begin{equation*}
(n-2)\ell \leq \lambda_i \leq (n-2)m \; \text{ for } 1 \leq i \leq n-1
\quad\text{ and }\quad
2(n-1)\ell \leq \lambda_n \leq 2(n-1)m.
\end{equation*}
Moreover, if $J$ is diagonally dominant, then the lower bounds still hold.
\end{lemma}
\begin{proof}
We first show that if $P \geq 0$ is a symmetric diagonally dominant matrix, then $P \succeq 0$. For any $x \in \R^n$,
\begin{equation*}
x^\top Px
= \sum_{i=1}^n P_{ii} x_i^2 + 2 \sum_{i < j} P_{ij} x_i x_j
\geq \sum_{i=1}^n \left( \sum_{j \neq i} P_{ij} \right) x_i^2 + 2 \sum_{i < j} P_{ij} x_i x_j
= \sum_{i < j} P_{ij} (x_i + x_j)^2
\geq 0.
\end{equation*}
Since the matrices $P = J-\ell S$ and $Q = mS-J$ are symmetric and diagonally balanced with nonnegative entries, it follows that $P,Q \succeq 0$ by the discussion above, which means $\ell S \preceq J \preceq mS$. The eigenvalues of $S$ are $\{n-2, \dots, n-2, 2(n-1)\}$, so the result follows by an application of~\cite[Corollary~7.7.4]{HornJohnson}. If $J$ is diagonally dominant, then $\ell S \preceq J$, and hence the lower bounds, still holds.
\end{proof}

\begin{lemma}\label{Lem:EigBalancedSubmatrix}
Let $J$ be a symmetric diagonally balanced $n \times n$ matrix with $0
< \ell \leq J_{ij} \leq m$ for $i \neq j$. For each $1 \leq i \leq n$,
let $J_{(i)}$ be the lower right $(n-i+1){\times}(n-i+1)$ block of $J$
as defined above, and suppose the eigenvalues of $J_{(i)}$ are  $\lambda_1 \leq \cdots \leq \lambda_{n-i+1}$. Then
\begin{equation*}
(n-2)\ell \leq \lambda_j \leq (n-2)m \; \text{ for } 1 \leq j \leq n-i
\quad\text{ and }\quad
(2n-i-1)\ell \leq \lambda_{n-i+1} \leq (2n-i-1)m.
\end{equation*}
Moreover, if $J$ is diagonally dominant, then the lower bounds still hold.
\end{lemma}
\begin{proof}
Write $J_{(i)} = H + D$, where $H$ is the $(n-i+1){\times}(n-i+1)$ diagonally balanced matrix and $D$ is diagonal with nonnegative entries. Note that $(i-1)\ell\eye \preceq D \preceq (i-1)m\eye$, so $(i-1)\ell\eye + H \preceq J_{(i)} \preceq (i-1)m\eye + H$. Thus by~\cite[Corollary~7.7.4]{HornJohnson} and by applying Lemma~\ref{Lem:EigBalanced} to $H$, we get, for $1 \leq j \leq n-i$,
\begin{equation*}
(n-2)\ell = (n-i-1)\ell + (i-1)\ell \leq \lambda_j \leq (n-i-1)m + (i-1)m = (n-2)m,
\end{equation*}
and for $j = n-i+1$,
\begin{equation*}
(2n-i-1)\ell = 2(n-i)\ell + (i-1)\ell \leq \lambda_{n-i+1} \leq 2(n-i)m + (i-1)m = (2n-i-1)m.
\end{equation*}
If $J$ is diagonally dominant, then $(i-1)\ell\eye + H \preceq J_{(i)}$ and hence the lower bounds still hold.
\end{proof}

\begin{proof_of}{Theorem~\ref{Thm:DetBound}}
For part {\bf  (a)}, suppose $J$ is diagonally dominant. For each $1
\leq i \leq n-1$ we have $J_{ii} \geq \sum_{j \neq i} J_{ij} \geq
b_{(i)}^\top \ones_{n-i}$, and by
Lemma~\ref{Lem:EigBalancedSubmatrix}, the maximum eigenvalue of
$J_{(i+1)}^{\,-1}$ is at most $\frac{1}{(n-2)\ell}$. Thus,
\begin{equation*}
\frac{b_{(i)}^{\top}J_{(i+1)}^{\,-1}b_{(i)}}{J_{ii}}
\leq \frac{1}{(n-2)\ell} \frac{b_{(i)}^{\top} b_{(i)}}{ J_{ii}}
\leq \frac{1}{(n-2)\ell} \: \frac{b_{(i)}^{\top} b_{(i)}}{b_{(i)}^\top \ones}
\leq \frac{\sqrt{(n-i+1)} \: m}{(n-2)\ell} \: \frac{\sqrt{b_{(i)}^{\top} b_{(i)}}}{b_{(i)}^\top \ones}.
\end{equation*}
Since each entry of $b_{(i)}$ is bounded by $\ell$ and $m$, the reverse Cauchy-Schwarz inequality~\cite[Ch.~5]{CSMC} gives us
\begin{equation*}
\frac{b_{(i)}^{\top}J_{(i+1)}^{\,-1}b_{(i)}}{J_{ii}} \leq \frac{\sqrt{(n-i+1)} \: m}{(n-2)\ell} \: \frac{\ell+m}{2\sqrt{\ell m(n-i+1)}} = \frac{1}{2(n-2)} \: \sqrt{\frac{m}{\ell}} \: \left(1 + \frac{m}{\ell}\right).
\end{equation*}
Substituting this inequality into~\eqref{Eq:HadamardProdInit} gives us the desired bound for part {\bf (a)}.

For part {\bf (b)}, suppose $J$ is diagonally balanced, so $J_{ii} \leq (n-1)m$ for each $1 \leq i \leq n-1$. By Lemma~\ref{Lem:EigBalancedSubmatrix} we know that the minimum eigenvalue of $J_{(i+1)}^{\,-1}$ is at least $\frac{1}{(2n-i-2)m}$, so
\begin{equation*}
\frac{b_{(i)}^{\top}J_{(i+1)}^{\,-1}b_{(i)}}{J_{ii}}
\geq \frac{1}{(2n-i-2)m} \: \frac{b_{(i)}^{\top} b_{(i)}}{J_{ii}}
\geq \frac{1}{2(n-1)m} \: \frac{(n-i+1) \ell^2}{(n-1)m}
= \frac{(n-i+1)\ell^2}{2(n-1)^2 m^2}.
\end{equation*}
Substituting this into~\eqref{Eq:HadamardProdInit} and using the inequality $1+x \leq \exp(x)$ gives us the desired bound:
\begin{flalign*}
&& \frac{\det(J)}{\prod_{i=1}^n J_{ii}}
\leq \prod_{i=1}^{n-1} \exp \left(-\frac{(n-i+1)\ell^2}{2(n-1)^2 m^2}\right)
= \exp\left(-\frac{(n+2)\ell^2}{4(n-1)m^2}\right)
\leq \exp\left(-\frac{\ell^2}{4m^2}\right). && \qed 
\end{flalign*}
\noqed
\end{proof_of}

We close this section with several examples.

\begin{example}\label{S_ex}
The matrix $S = (n-2)\eye_n + \ones_n \ones_n^{\top}$ has eigenvalues $\{n-2,\dots,n-2,2(n-1)\}$, so
\begin{flalign*}
&&\frac{\det(S)}{\prod_{i=1}^n S_{ii}} = \frac{2(n-2)^{n-1}(n-1)}{(n-1)^{n}} = 2\left(1- \frac{1}{n-1}\right)^{n-1} \to  \frac{2}{e}
\quad \text{ as } n \to \infty. && \qed
\end{flalign*}
\end{example}

\begin{example}
When $J$ is strictly diagonally dominant, the ratio
$\det(J)/\prod_{i=1}^n J_{ii}$ can be arbitrarily close to $1$. For
instance, consider $J = \alpha\eye_n + \ones_n \ones_n^\top$ with
$\alpha \geq n-2$, which has eigenvalues $\{(n+\alpha), \alpha, \dots, \alpha\}$ so
\begin{flalign*}
&& \frac{\det(J)}{\prod_{i=1}^n J_{ii}} = \frac{(n+\alpha)
  \alpha^{n-1}}{(\alpha+1)^n} \to 1 \quad \text{ as } \alpha \to
\infty. &&\qed
\end{flalign*}
\end{example}

\begin{example}\label{NonupperBndEx}
The following example demonstrates that we need an upper bound on the entries of $J$ in Theorem \ref{Thm:DetBound} ({\bf a}). Let $n = 2k$ for some $k \in \N$, and consider the matrix $J$ in the following block form:
\begin{equation*}
J = \begin{pmatrix} A & B \\ B & A \end{pmatrix}, \quad A =
(km+k\ell-2\ell)\eye_k + \ell \ones_k\ones_k^\top,\,\, B = m\ones_k\ones_k^\top.
\end{equation*}
By the determinant block formula (since $A$ and $B$ commute), we have
\begin{equation*}
\begin{split}
\det(J) &= \det(A^2-B^2) \\
&= \det \Big[ (km+k\ell-2\ell)^2 \eye_k + (2k\ell m+3k\ell^2-4\ell^2 - km^2) \ones_k\ones_k^\top \Big] \\
&= 4\ell(k-1)(km+k\ell-\ell) \cdot
    (km+k\ell-2\ell)^{2k-2},
\end{split}
\end{equation*}
where the last equality is obtained by considering the eigenvalues of $A^2-B^2$. Then
\begin{equation*}
\begin{split}
\frac{\det(J)}{\prod_{i=1}^n J_{ii}}
&= \frac{4\ell(k-1)(km+k\ell-\ell) \cdot
    (km+k\ell-2\ell)^{2k-2}}{\big(km+k\ell-\ell\big)^{2k}} \to \frac{4\ell}{\ell+m} \exp\left(-\frac{2\ell}{\ell+m}\right) \ \  \text{as } k \to \infty.
\end{split}
\end{equation*}
Note that the last quantity above tends to $0$ as $m/\ell \to
\infty$. \qed
\end{example}

\section{Open Problems}\label{Sec:Problems}

As an analogue to Theorem~\ref{Thm:GeneralMain}, we also conjecture a tight lower bound for symmetric diagonally
dominant matrices $J > 0$ whose off-diagonal entries
and diagonal dominances are bounded above. Observe that when $J_{ij} \leq
m$ and $\Delta_i(J) \leq \delta$ for all $i \neq j$, then $J \leq
(m(n-2)+\delta) \eye_n + m \ones_n \ones_n^\top$.
\begin{conjecture} Let $n \geq 3$ and let $S(\alpha,m) = \alpha \eye_n
  + m \ones_n \ones_n^\top$. For all $n \times n$ symmetric diagonally
  dominant matrices $0 < J \leq S(\alpha,m)$, we have
  \begin{align*}
    \|J^{-1} \|_\infty \geq \|S(\alpha,m)^{-1} \|_\infty =
    \frac{\alpha + 2m(n-1)}{\alpha(\alpha+mn)}.
  \end{align*}
Moreover, equality is achieved if and only if $J = S(\alpha,m)$.
\end{conjecture}
In working towards the proof of this conjecture, the following problem
may be useful. 
\begin{problem}
Given a signless Laplacian $P$ of a graph $G$, give an exact
combinatorial formula for the entries of $(S+tP)^{-1}$ for any $t >
0$. More precisely, since each entry of $(S+tP)^{-1}$ is a rational function of
$t$, derive a formula for the coefficients of this rational function in
terms of the combinatorics of the graph $G$.
\end{problem}

We also conjecture that a dependence on the largest entry can be removed in Theorem \ref{Thm:DetBound} (\textbf{b}).

\begin{conjecture}
For a positive, diagonally balanced symmetric $J$, we have the bound:
\[ \frac{\det(J)}{\prod_{i=1}^n J_{ii}} \leq \frac{\det(S)}{(n-1)^{n}} = 2\left(1-\frac{1}{n-1}\right)^{n-1} \to \frac{2}{e}.\]
\end{conjecture}



\begin{thebibliography}{9}

\bibitem{diaconis2011}
S.~Chatterjee, P.~Diaconis, A.~Sly. \textit{Random graphs with a given degree sequence.} Annals of Applied Probability, \textbf{21} (2011), 1400--1435.

\bibitem{cvetkov2010towards3}
D.~Cvetkovi{\'c} and S.~Simi{\'c}. \textit{Towards a spectral theory of graphs based on the signless Laplacian, III.} Applicable Analysis and Discrete Mathematics, 4(1):156--166, 2010.

\bibitem{demmel1992block}
J.W.~Demmel, N.J.~Higham, and R.~Schreiber. \textit{Block LU factorization.} Research Institute of Advanced Computer Science, 1992.

\bibitem{HilWib}
C.~Hillar and A.~Wibisono. {\em Maximum entropy distributions on graphs.}, arXiv:1301.3321, 2013.
 
\bibitem{HornJohnson}
R.A.~Horn and C.R.~Johnson. {\em Matrix Analysis.} Cambridge University Press, 1990.

\bibitem{HornJohnsonTopic}
R.A.~Horn and C.R.~Johnson. {\em Topics in Matrix Analysis.} Cambridge University Press, 1991.

\bibitem{HJ02}
C.R.~Johnson and C.~Hillar. \textit{Eigenvalues of words in two positive definite letters.} SIAM Journal on Matrix Analysis and Applications, \textbf{23} (2002), 916--928.

\bibitem{Lang}
S.~Lang. {\em Real and Functional Analysis.} Springer, 1993.

\bibitem{li2008}
W.~Li.  {\em The infinity norm bound for the inverse of nonsingular diagonal dominant matrices.} Applied Mathematics Letters, \textbf{21} (2008), 258--263.

\bibitem{Munkres}
J.R.~Munkres. {\em Topology (2nd Edition).} Prentice Hall, 2000.

\bibitem{SturmEntDisc}
R.~Sanyal, B.~Sturmfels, C.~Vinzant. \textit{The entropic discriminant.} Advances in Mathematics, to appear.

\bibitem{shivakumar1996two}
P.N.~Shivakumar, J.J.~Williams, Q.~Ye, and C.A.~Marinov. \textit{On two-sided bounds related to weakly diagonally dominant $M$-matrices with application to digital circuit dynamics.}
SIAM Journal on Matrix Analysis and Applications, \textbf{17} (1996) 298.

\bibitem{spielman10}
D.A.~Spielman. \textit{Algorithms, graph theory, and linear equations in {L}aplacian matrices.} Proceedings of the international congress of mathematicians, \textbf{4} (2010), 2698--2722.

\bibitem{spielman-teng04}
D.A.~Spielman and S.H.~Teng. \textit{Nearly-linear time algorithms for graph partitioning, graph sparsification, and solving linear systems.} Proceedings of the thirty-sixth annual ACM symposium on Theory of computing (STOC '04), (2004), 81--90.

\bibitem{spielman-teng06}
D.A.~Spielman and S.H.~Teng. \textit{Nearly-linear time algorithms for preconditioning and solving symmetric, diagonally dominant linear systems.} http://arxiv.org/abs/0808.4134, 2006.

\bibitem{CSMC}
J.M.~Steele. \textit{The Cauchy-Schwarz Master Class.} Cambridge University Press, 2004.

\bibitem{varah1975lower}
J.M.~Varah. \textit{A lower bound for the smallest singular value of a matrix.} Linear Algebra and its Applications,
\textbf{11} (1975), 3--5.

\bibitem{varga1976diagonal}
R.S.~Varga. \textit{On diagonal dominance arguments for bounding $\|A^{-1}\|_{\infty}$.} Linear Algebra and its Applications, \textbf{14} (1976), 211--217.

\end{thebibliography}
\end{document}